\PassOptionsToPackage{a4paper,margin=25mm}{geometry}

\RequirePackage{iftex}
\ifpTeX
  \PassOptionsToPackage{dvipdfmx}{graphicx}
  \PassOptionsToPackage{dvipdfmx}{hyperref}
\fi

\documentclass[12pt]{article}
\usepackage{da-common}
\usepackage{paper2-h4}
\usetikzlibrary{arrows.meta}

\title{Affine Angles via Area Cross Ratio and Isoptic Hyperbolas}
\author{%
Masanori Nakazato\\
Mita International School of Science, Tokyo, Japan\\
\texttt{masa727axio.math@gmail.com}\\[6pt]
\small Formerly at the Graduate School of Science, Tohoku University (M.Sc.)
}
\date{March 24, 2026}

\begin{document}
	\maketitle

\begin{abstract}
Affine geometry is usually regarded as a framework in which metric notions such
as distance and angle are absent. However, just as projective geometry produces
various metric geometries by introducing additional structures on the line at infinity,
affine geometry can also serve as a natural basis for an angular geometry once certain
directions at infinity are fixed.

In this paper we introduce an affine angle determined by two fixed directions on
the line at infinity and defined via an area cross ratio. This quantity is invariant under
affine transformations preserving the chosen directions.

We show that the locus of points from which a fixed segment is seen under a constant
affine angle is a hyperbola whose asymptotes are parallel to the chosen directions.
This provides an affine analogue of the classical fact that in Euclidean geometry the
isoptic curve of a segment is a circle.

Furthermore, we establish that this angle arises as a parabolic degeneration of the
Cayley--Klein angle, and that the same quantity naturally appears in a power theorem
associated with hyperbolas.

These results provide a unified perspective linking affine angles, isoptic hyperbolas,
and hyperbolic power through the area cross ratio.
\end{abstract}

\subsection*{Notation}
We summarize the main notation used throughout this paper.

\begin{itemize}
\item $U, V$ : lines passing through the origin $O$ with directions $u$ and $v$, respectively.
\item $\Lambda$ : an auxiliary line intersecting both $U$ and $V$.
\item $P_L$ : for a line $L$, we denote $P_L := L \cap \Lambda$.
      In particular, $P_U := U \cap \Lambda$ and $P_V := V \cap \Lambda$.
\item $\sigma_\Lambda(L)$ : a quantity defined via a ratio of triangle areas.
\item $\mathrm{CR}_{\mathrm{area}}(L_1,L_2;R_1,R_2)$ : the area cross ratio.
\item $\phi_{u,v}(O;A,B)$ : the $(u,v)$-affine angle.
\end{itemize}

\section{Introduction}

\subsection{Motivation}

In this paper, we reconsider the notion of angle in geometry.
Since affine transformations do not, in general, preserve distances or angles,
it has been commonly believed that a natural notion of angle cannot be defined
within affine geometry.

However, this viewpoint relies on the Euclidean notion of angle.
In this work, we instead construct an affine-invariant quantity
and define an angle based on it.

More precisely, we introduce an angle defined via the area cross ratio,
which is invariant under the subgroup of affine transformations preserving
two fixed directions $u,v$ on the line at infinity:
\[
\phi_{u,v}(O;A,B)
=
\frac{1}{2}\log \mathrm{CR}_{\mathrm{area}}(L_A,L_B;U,V).
\]
We then study the geometric properties of this quantity.

\subsection{Main Results}

The main result of this paper is summarized in the following theorem.
The most striking geometric feature of this angle is that
the locus of points from which a fixed segment is seen under a constant affine angle
is completely described by a hyperbola.

\begin{maintheorem}[Fundamental Theorem of Affine Angle Geometry]
Fix two directions $u,v$ on the line at infinity,
and let $U,V$ be lines parallel to them.
For two rays $OA, OB$ with vertex $O$, define
\[
\phi_{u,v}(O;A,B)
=
\frac{1}{2}\log \mathrm{CR}_{\mathrm{area}}(L_A,L_B;U,V).
\]

Then this quantity satisfies the fundamental properties of an angle,
namely antisymmetry, additivity, vanishing, scaling invariance,
and continuity, and is invariant under the subgroup of affine transformations
preserving the directions $u,v$.

Moreover, the following hold:
\begin{enumerate}
\item This angle arises naturally as a parabolic degeneration of the Cayley--Klein angle.
\item For fixed points $A,B$, the locus of points $P$ satisfying
\[
\phi_{u,v}(P;A,B)=\theta
\]
is a hyperbola whose asymptotes are parallel to $u$ and $v$.
\end{enumerate}
\end{maintheorem}

We first introduce the area cross ratio as an affine invariant,
and define the affine angle based on it.
We then verify that this quantity satisfies the axioms of angle.
Next, we demonstrate that the corresponding invariant transformation group
and the isoptic loci naturally lead to hyperbolas.
Finally, we explain that this angle can be understood
as a parabolic degeneration in Cayley--Klein geometry
(cf.~\cite{Klein1871, Klein1873}),
and establish a corresponding power theorem for hyperbolas.

\subsection{Position of the Affine Angle}

The axiom system A1--A5 for angles introduced in Nakazato~\cite{Nakazato2025Base1}
is based on the axiomatic approach originating from Hilbert's geometry~\cite{Hilbert1899}.
However, these axioms do not specify a unique geometric realization;
rather, any quantity satisfying them can serve as an angle
in different geometric settings.

In difference-angle geometry, this axiom system is realized
by the difference of slopes.
The affine angle introduced in this paper provides another realization
of the same axioms.

This affine angle is inspired by the cross ratio in projective geometry
and is constructed from an affine invariant based on area ratios.
Although it is not invariant under the full affine group,
it is invariant under the subgroup preserving two directions
on the line at infinity,
which is formally analogous to the invariance of Euclidean angles.

Moreover, the corresponding isoptic loci appear as hyperbolas.
This parallels the classical facts that isoptic loci are circles
in Euclidean geometry and parabolas in difference-angle geometry.

Thus, under the same axiom system for angles,
different realizations give rise to different conic sections,
showing that the axioms of angle provide a framework
allowing multiple geometric realizations rather than determining a single geometry.

\subsection{The Affine Plane as a Geometric Platform}

Projective geometry provides a framework in which many geometric properties
can be described without introducing a metric.
By equipping the projective plane with an absolute conic
and considering the group of projective transformations preserving it,
one obtains Cayley--Klein geometries such as Euclidean and hyperbolic geometry.

On the other hand, affine geometry is obtained
by selecting a line in the projective plane
and designating it as the line at infinity $\ell_\infty$.
The projective transformations preserving $\ell_\infty$
form the affine transformation group,
which can be expressed as linear transformations combined with translations.

Affine geometry studies properties invariant under this group.

If one can define an angle based on suitable affine invariants
within this framework,
then affine geometry provides a natural platform
for constructing a new angular geometry.

\begin{figure}[htbp]
  \centering
  \begin{minipage}{0.80\textwidth}
    \centering
    \includegraphics[width=\linewidth]{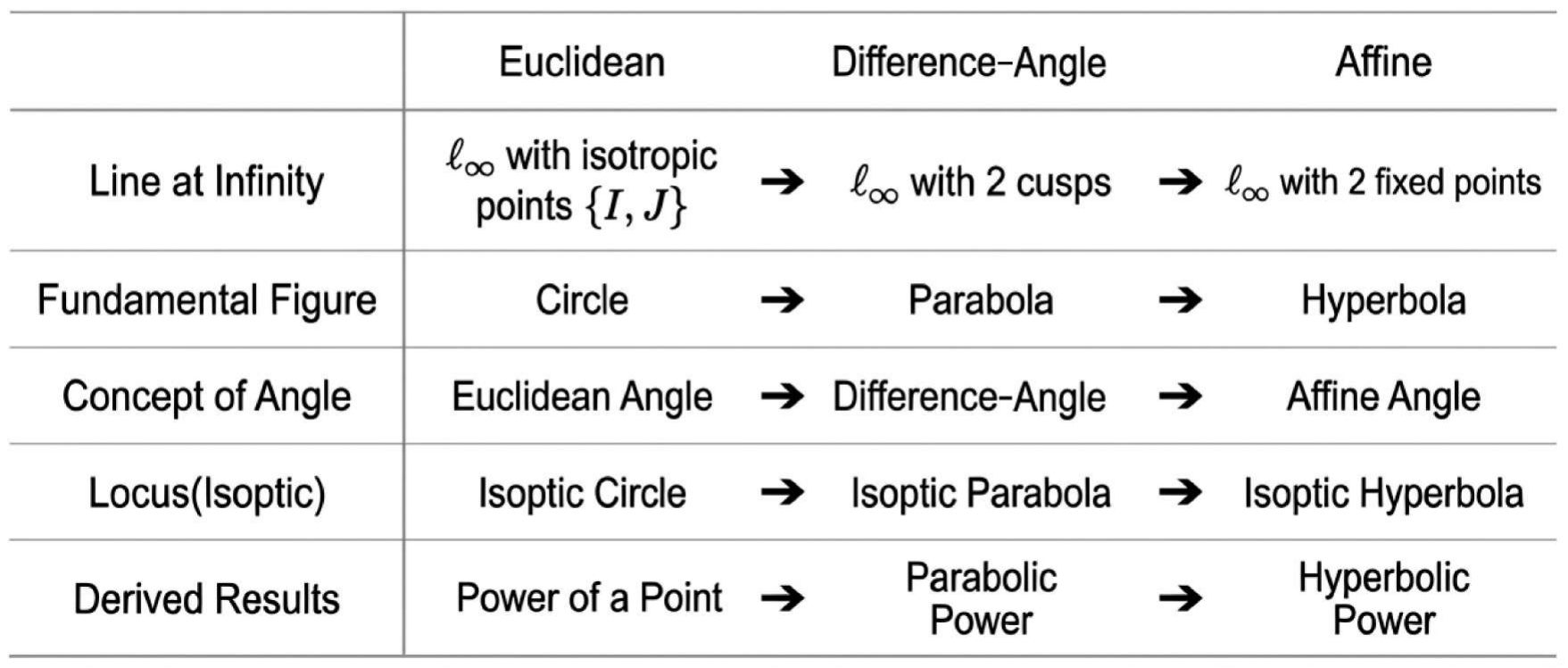}
    \subcaption{}
    \label{fig:geometry-correspondence}
  \end{minipage}
\caption{
Correspondence between circle, parabola, and hyperbola geometries.
Each geometry admits an angle, an isoptic locus of a segment,
and an associated power theorem, forming a unified structure.
}
\end{figure}

\section{Affine Angle}

In this section, we first introduce a quantity $\sigma_\Lambda(L)$
associated with a ray $L$ using an auxiliary line $\Lambda$.
We then define the area cross ratio $CR_{\mathrm{area}}$
as a quantity independent of $\Lambda$.
Finally, the affine angle is defined as the logarithm of this invariant.

Throughout this paper, we fix an ordered pair of directions $(u,v)$.
Let $U$ and $V$ be the lines through a point $O$
with directions $u$ and $v$, respectively.
These lines $U$ and $V$ serve merely as carriers of the directions,
and the essential structure of the angle depends only on the ordered pair $(u,v)$.

\begin{figure}[htbp]
  \centering
 \begin{subfigure}{0.32\textwidth}
    \centering
  \includegraphics[scale=0.6]{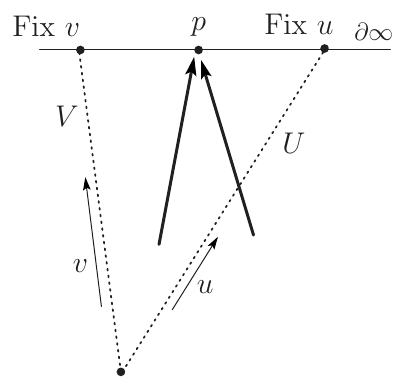}
  \caption{}
  \label{fig:affine-angle-ivt}
  \end{subfigure}
  \begin{subfigure}{0.32\textwidth}
  \centering
  \includegraphics[scale=0.6]{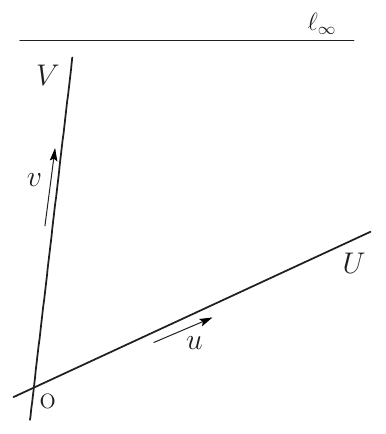}
  \caption{}
  \label{fig:u-v-fix}
  \end{subfigure}
  \begin{subfigure}{0.32\textwidth}
  \centering
  \includegraphics[scale=0.6]{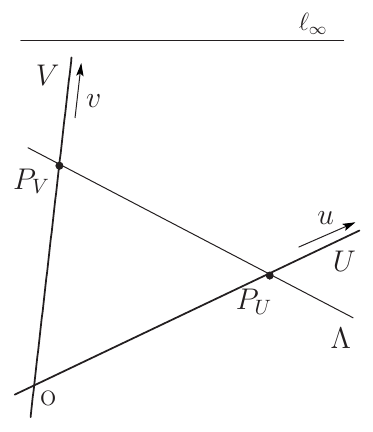}
  \caption{}
  \label{fig:Lambda-cross-points}
  \end{subfigure}
\caption{(a) Affine angle represented on the line at infinity. 
(b) Affine plane with fixed directions $u$ and $v$ defining the angle.}
\end{figure}

\bigskip

We denote by $[XYZ]$ the signed area of the triangle $XYZ$.
That is, if $\vec{XY}=(p,q)$ and $\vec{XZ}=(r,s)$, then
\[
[XYZ]=\frac{1}{2}(ps-qr).
\]

\begin{definition}[$\Lambda$-Area Ratio]
Let $\Lambda$ be a line intersecting both $U$ and $V$,
and define $P_U=\Lambda\cap U$, $P_V=\Lambda\cap V$.
For a ray $L$ emanating from $O$, let $P_L=L\cap\Lambda$,
assuming $P_L\neq P_U,P_V$.
We define
\[
\sigma_\Lambda(L):=\frac{[OP_LP_U]}{[OP_LP_V]}
\]
and call it the $\Lambda$-area ratio.
\end{definition}

\begin{figure}[htbp]
  \centering
 \begin{subfigure}{0.32\textwidth}
    \centering
  \includegraphics[scale=0.6]{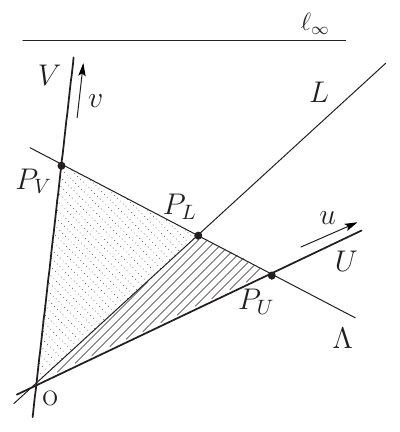}
  \caption{}
  \label{fig:Lambda-area-ratio-on-L}
  \end{subfigure}
  \begin{subfigure}{0.32\textwidth}
  \centering
  \includegraphics[scale=0.6]{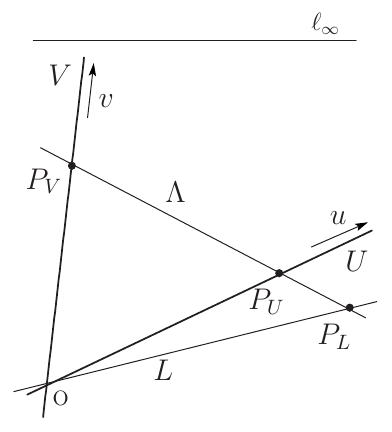}
  \caption{}
  \label{fig:positive-sigma-lambda-L}
  \end{subfigure}
  \begin{subfigure}{0.32\textwidth}
  \centering
  \includegraphics[scale=0.6]{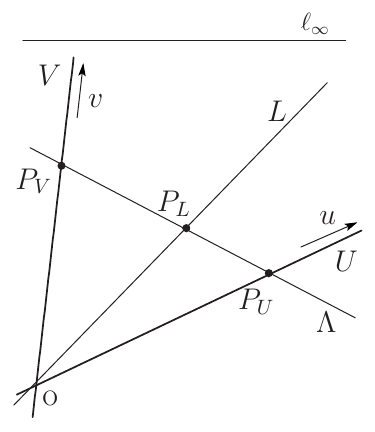}
  \caption{}
  \label{fig:negative-sigma-lambda-L}
  \end{subfigure}
\caption{(a) Definition of the $\Lambda$-area ratio.
(b) Case $\sigma_\Lambda(L) > 0$.
(c) Case $\sigma_\Lambda(L) < 0$.}
\end{figure}

\begin{lemma}[Sign of $\sigma_\Lambda(L)$]\label{lem:sign-sigma-Lambda}
The set $\Lambda\setminus\{P_U,P_V\}$ is divided into three connected components,
and the sign of $\sigma_\Lambda(L)$ is constant on each component.
More precisely:
\begin{itemize}
\item If $P_L\in (P_U,P_V)$, then $\sigma_\Lambda(L)<0$.
\item If $P_L\in \Lambda\setminus [P_U,P_V]$, then $\sigma_\Lambda(L)>0$.
\end{itemize}
\end{lemma}

\begin{proof}
Fix an orientation of the affine plane.
If $P_L\in (P_U,P_V)$, then the points $P_U, P_L, P_V$ are collinear,
with $P_L$ lying between $P_U$ and $P_V$.
Hence the signed areas $[OP_LP_U]$ and $[OP_LP_V]$
have opposite signs, and therefore
\[
\sigma_\Lambda(L)<0.
\]

If $P_L\in \Lambda\setminus [P_U,P_V]$, then
$[OP_LP_U]$ and $[OP_LP_V]$ have the same sign,
and thus $\sigma_\Lambda(L)>0$.
\end{proof}

Although $\sigma_\Lambda(L)$ depends on $\Lambda$ and $L$,
the signed areas are scaled by a common factor under affine transformations,
so their ratio is preserved.
Therefore, this quantity is affine-natural
and serves as the basis for constructing the area cross ratio
and the affine angle.

\begin{definition}[Area Cross Ratio]
For four rays $L_1, L_2, R_1, R_2$ emanating from $O$,
consider the cross ratio
\[
cr\bigl(\sigma_\Lambda(L_1),\ \sigma_\Lambda(L_2);\ 
\sigma_\Lambda(R_1),\ \sigma_\Lambda(R_2)\bigr).
\]
We define the area cross ratio by
\[
CR_{\mathrm{area}}(L_1,L_2;R_1,R_2)
:=
cr\bigl(\sigma_\Lambda(L_1),\ \sigma_\Lambda(L_2);\ 
\sigma_\Lambda(R_1),\ \sigma_\Lambda(R_2)\bigr).
\]
\end{definition}

\begin{theorem}[$CR_{\mathrm{area}}$ is independent of $\Lambda$]\label{thm:CR-Lambda-property}
The quantity $CR_{\mathrm{area}}(L_1,L_2;R_1,R_2)$
does not depend on the choice of $\Lambda$.
\end{theorem}

\begin{proof}
By applying a suitable affine transformation $F$, we normalize the configuration
so that $U$ becomes the $x$-axis and $V$ becomes the $y$-axis.
Let $\Lambda$ be given by $ax+by=1$ with $a>0$, $b>0$.

Let a ray $L$ be given by $y=mx$ with $m>0$.
Then
\[
P_U=\left(\frac{1}{a},0\right), \quad
P_V=\left(0,\frac{1}{b}\right), \quad
P_L=\left(\frac{1}{a+bm},\frac{m}{a+bm}\right).
\]
Hence,
\[
[OP_LP_U]=\frac{m}{2a(a+bm)}, \quad
[OP_LP_V]=\frac{1}{2b(a+bm)}.
\]
Therefore,
\[
\sigma_\Lambda(L)
=
\frac{\frac{m}{2a(a+bm)}}{\frac{1}{2b(a+bm)}}
=
\frac{b}{a}m.
\]
Note that $\frac{b}{a}$ depends only on $\Lambda$ and not on $L$.

Next, for another line $\Lambda'$ given by $a'x+b'y=1$ with $a'>0$, $b'>0$,
we obtain
\[
\sigma_{\Lambda'}(L)=\frac{b'}{a'}m,
\]
and hence
\[
\sigma_{\Lambda'}(L)=c\,\sigma_\Lambda(L),
\qquad
\left(c=\frac{b'/a'}{b/a}=\mathrm{const.}\right).
\]

It follows that
\begin{align*}
cr\bigl(\sigma_{\Lambda'}(L_1),\sigma_{\Lambda'}(L_2);\sigma_{\Lambda'}(R_1),\sigma_{\Lambda'}(R_2)\bigr)
&=
cr\bigl(c\sigma_{\Lambda}(L_1),c\sigma_{\Lambda}(L_2);
c\sigma_{\Lambda}(R_1),c\sigma_{\Lambda}(R_2)\bigr)\\
&=
cr\bigl(\sigma_{\Lambda}(L_1),\sigma_{\Lambda}(L_2);
\sigma_{\Lambda}(R_1),\sigma_{\Lambda}(R_2)\bigr).
\end{align*}
\end{proof}

As boundary cases, if $L\parallel U$, then $\sigma_\Lambda(L)=0$.
If $L\parallel V$, we interpret
\[
\sigma_\Lambda(L)\to +\infty \qquad (\mathrm{slope}(L)\to \mathrm{slope}(V))
\]
as a limit.
These cases can be treated naturally by extending the cross ratio to limits,
and the scaling invariance is preserved.

\begin{theorem}[Affine invariance of $CR_{\mathrm{area}}$]
For any affine transformation $T$, we have
\[
CR_{\mathrm{area}}\bigl(T(L_1),T(L_2);T(R_1),T(R_2)\bigr)
=
CR_{\mathrm{area}}(L_1,L_2;R_1,R_2).
\]
\end{theorem}

\begin{proof}
An affine transformation $T$ multiplies any signed area by $|\det(T)|$.
Hence this factor cancels in all area ratios appearing in $CR_{\mathrm{area}}$,
which proves the invariance.
\end{proof}

\begin{theorem}\label{thm:CR-Lambda-property2}
If $R_1=U$ and $R_2=V$, then
\[
\frac{\sigma_{\Lambda}(L_1)}{\sigma_{\Lambda}(L_2)}
\]
is independent of $\Lambda$.
\end{theorem}

It is well known that the cross ratio is characterized
by the existence of a projective transformation sending any three points
to $\infty,0,1$ (see, for example, \cite{Kobayashi1990}).

\begin{proof}
Let $R_1=U$ and $R_2=V$.
By definition,
\[
CR_{\mathrm{area}}(L_1,L_2;U,V)
=
cr\bigl(\sigma_\Lambda(L_1),\sigma_\Lambda(L_2);
\sigma_\Lambda(U),\sigma_\Lambda(V)\bigr).
\]

Here $\sigma_\Lambda(U)=0$, and $\sigma_\Lambda(V)$ diverges.
Thus we regard $\sigma_\Lambda(V)=\infty$ in the extended real line
$\overline{\mathbb R}=\mathbb R\cup\{\infty\}$.

Using the identity
\[
cr(x,y;0,\infty)=\frac{x}{y},
\]
we obtain
\[
CR_{\mathrm{area}}(L_1,L_2;U,V)
=
\frac{\sigma_\Lambda(L_1)}{\sigma_\Lambda(L_2)}.
\]

On the other hand, by Theorem~\ref{thm:CR-Lambda-property},
the left-hand side is independent of $\Lambda$,
and hence so is the right-hand side.
The boundary cases $m\to0$ ($L\parallel U$) and
$m\to\infty$ ($L\parallel V$) are interpreted in the same limiting sense.
\end{proof}

From now on, when no confusion arises, we write
\[
\sigma(L):=\sigma_{\Lambda}(L).
\]

The set $\Lambda\setminus\{P_U,P_V\}$ splits into two connected components.
Fix one such component $\gamma$.
For all $L$ such that $P_L \in \gamma$, the sign of $\sigma(L)$ is constant.
Therefore, if $L_A$ and $L_B$ belong to the same component $\gamma$,
then $\sigma(L_A)$ and $\sigma(L_B)$ have the same sign, and hence
\[
\frac{\sigma(L_A)}{\sigma(L_B)}>0.
\]

Consequently, the area cross ratio
\[
CR_{\mathrm{area}}(L_A,L_B;U,V)
=
\frac{\sigma(L_A)}{\sigma(L_B)}
\]
is a positive real number.

Taking the logarithm, we obtain an additive quantity
\[
\log\frac{\sigma(L_A)}{\sigma(L_B)}.
\]
This quantity serves as the fundamental invariant,
and after normalization, defines the affine angle.

\begin{definition}[Affine Angle via the Area Cross Ratio]
\label{def:affine-angle-acr}
Let $L_A=OA$ and $L_B=OB$.
We define
\[
\phi_{u,v}(O;A,B)
:=
\frac{1}{2}\log\!\left(
CR_{\mathrm{area}}(L_A,L_B;U,V)
\right)
=
\frac{1}{2}\log\!\frac{\sigma(L_A)}{\sigma(L_B)}.
\]
We call this quantity the $(u,v)$-affine angle.
\end{definition}

\begin{remark}
The factor $1/2$ is introduced for two reasons.
First, it ensures consistency with the parabolic degeneration
of the Cayley--Klein angle.
Second, it simplifies the expressions of the corresponding isoptic curves.
\end{remark}

The $(u,v)$-affine angle $\phi_{u,v}(O;A,B)$
is defined on a generally disconnected domain.
Therefore, in what follows,
we restrict our attention to the case where
$P_{L_A}$ and $P_{L_B}$ belong to the same connected component of
$\Lambda\setminus\{P_U,P_V\}$,
and treat $\phi_{u,v}(O;A,B)$ as a real-valued angle only in this setting.

Under this restriction,
the angle is uniquely defined,
and satisfies antisymmetry and additivity as an oriented angle.
This means that the axiom A5 introduced in the next section
(the continuity and uniqueness of angle)
is naturally satisfied in the present framework.

\begin{proposition}[Reality of the $(u,v)$-affine angle]
\label{prop:real-uv-affine-angle}
Assume that $P_{L_A},P_{L_B}\notin\{P_U,P_V\}$.
Then:
\begin{itemize}
\item If $P_{L_A}$ and $P_{L_B}$ belong to the same connected component of
$\Lambda\setminus\{P_U,P_V\}$, then
\[
\phi_{u,v}(O;A,B)\in\mathbb R.
\]

\item If they belong to different connected components, then
\[
\phi_{u,v}(O;A,B)\notin\mathbb R.
\]
\end{itemize}
\end{proposition}

\begin{proof}
By Lemma~\ref{lem:sign-sigma-Lambda},
if $P_{L_A}$ and $P_{L_B}$ belong to the same connected component,
then $\sigma(L_A)$ and $\sigma(L_B)$ have the same sign,
and hence
\[
\frac{\sigma(L_A)}{\sigma(L_B)}>0.
\]
Therefore the logarithm is well defined as a real number,
and $\phi_{u,v}(O;A,B)\in\mathbb R$.

On the other hand, if $P_{L_A}$ and $P_{L_B}$
belong to different components,
then $\frac{\sigma(L_A)}{\sigma(L_B)}<0$,
and thus $\phi_{u,v}(O;A,B)$ is not real-valued.
\end{proof}

\begin{remark}
Thus, in order to treat the $(u,v)$-affine angle as a real-valued quantity,
it suffices to restrict to a single connected component of
$\Lambda\setminus\{P_U,P_V\}$.
Under this restriction, the argument of the logarithm is always positive,
and the angle is uniquely defined as a real number.
\end{remark}

\section{Validity of the Affine Angle}

The $(u,v)$-affine angle $\phi_{u,v}(O;A,B)$
can be regarded as a concrete realization of a parabolic degeneration.
In this section, we verify directly that the quantity
$\phi_{u,v}(O;A,B)$ defined above satisfies
the axioms of angle (A1--A5) introduced in Nakazato~\cite{Nakazato2025Base1}.
This establishes that the present angle is well-defined.
For the axioms themselves, we refer to \cite{Nakazato2025Base1}.

\begin{definition}[Domain and Singular Set]
Fix a base point $P=O$ and choose two non-parallel reference directions $u,v$,
represented by rays $U,V\in\mathcal R_P$.
For each ray $L\in\mathcal R_P$, define
\[
\sigma(L):=\frac{[OP_LP_U]}{[OP_LP_V]},
\]
where the sign is taken consistently with the orientation from $U$ to $V$.

Define the singular set
\[
S_P:=\{L \mid L\parallel U \text{ or } L\parallel V\},
\]
and let $D_P:=\mathcal R_P\setminus S_P$ be the domain.

On each connected component of $D_P$, we have $\sigma(L)>0$,
and we define the affine angle by
\[
\phi_{u,v}(O;A,B)
:=
\frac{1}{2}\log\!\frac{\sigma(L_A)}{\sigma(L_B)}.
\]
The values lie in $\mathbb{R}$.
\end{definition}

\subsection{Verification of A1--A4}

We now verify that the affine angle satisfies A1--A4.
We begin with the following lemma, which is needed for A3.

\begin{lemma}[Injectivity on each component]\label{lem:sigma_injective}
On each connected component $\gamma$ of $D_P$,
the map $L\mapsto\sigma(L)$ is continuous and strictly monotone.
In particular, it is injective on $\gamma$.
\end{lemma}

\begin{proof}
As in the previous section, we normalize the configuration
by a suitable affine transformation so that $P=O$,
$U$ is the $x$-axis, and $V$ is the $y$-axis.
Let $\Lambda$ be given by
\[
\Lambda:\ ax+by=1 \qquad (a>0,\ b>0).
\]

Let $L$ be a ray through $O$ with $L\not\parallel U,V$,
and write its slope as $m=\mathrm{slope}(L)\in\mathbb{R}\setminus\{0\}$.
Then, from the previous computation,
\[
\sigma_\Lambda(L)=\frac{b}{a}\,m.
\]

Each connected component $\gamma$ of $D_P=\mathcal R_P\setminus S_P$
is identified with either $(0,\infty)$ or $(-\infty,0)$
via the slope parameter.
Hence $m$ varies continuously with fixed sign on $\gamma$.

Since $\frac{b}{a}>0$ is constant,
the map $L\mapsto\sigma_\Lambda(L)$ is continuous
and strictly monotone with respect to $m$.
Therefore it is injective on $\gamma$.
\end{proof}

\begin{proposition}[Satisfaction of A1--A4]
The affine angle $\phi_{u,v}$ defined above satisfies A1--A4.
\end{proposition}

\begin{proof}
\textbf{A1 (antisymmetry):}
\[
\phi_{u,v}(O;A,B)
=
\frac{1}{2}\log\frac{\sigma(L_A)}{\sigma(L_B)}
=
-\frac{1}{2}\log\frac{\sigma(L_B)}{\sigma(L_A)}
=
-\phi_{u,v}(O;B,A).
\]

\smallskip
\textbf{A2 (additivity):}
\[
\phi_{u,v}(O;A,C)
=
\frac{1}{2}\log\frac{\sigma(L_A)}{\sigma(L_C)}
=
\frac{1}{2}\log\frac{\sigma(L_A)}{\sigma(L_B)}
+
\frac{1}{2}\log\frac{\sigma(L_B)}{\sigma(L_C)}
=
\phi_{u,v}(O;A,B)+\phi_{u,v}(O;B,C).
\]

\smallskip
\textbf{A3 (vanishing):}
\[
\phi_{u,v}(O;A,B)=0 \iff \sigma(L_A)=\sigma(L_B).
\]
Since $L_A$ and $L_B$ lie in the same connected component,
injectivity of $\sigma$ (Lemma~\ref{lem:sigma_injective})
implies $L_A=L_B$.
Thus $OA$ and $OB$ coincide, and hence $O,A,B$ are collinear
(with the order condition as in A3).

\smallskip
\textbf{A4 (scaling invariance):}
If $A'\in\overrightarrow{OA}$ and $B'\in\overrightarrow{OB}$,
then
\[
L_{A'}=L_A,\qquad L_{B'}=L_B.
\]
Hence
\[
\sigma(L_{A'})=\sigma(L_A),\qquad
\sigma(L_{B'})=\sigma(L_B),
\]
and therefore
\[
\phi_{u,v}(O;A',B')=\phi_{u,v}(O;A,B).
\]
\end{proof}

\subsection{Verification of A5 and Interpretation as a Cayley--Klein Angle}

Let $r,s,t\in\mathcal R_P$ be rays with common vertex $P$.

\begin{lemma}[A5(i): existence and uniqueness of the midpoint subdivision]
For any connected component $\gamma$ of $D_P$ and any $r,s\in \gamma$,
there exists a unique $t\in \gamma$ such that
\[
\phi_{u,v}(O;r,t)=\phi_{u,v}(O;t,s)=\frac{1}{2}\,\phi_{u,v}(O;r,s).
\]
\end{lemma}

\begin{proof}
Since $\sigma(r)>0$,  $\sigma(s)>0$,  $\sigma(t)>0$ on $\gamma$, we have
\[
\phi_{u,v}(O;r,t)=\frac{1}{2}\,\phi_{u,v}(O;r,s)
\iff
\log\!\frac{\sigma(r)}{\sigma(t)}
=\frac{1}{2}\log\!\frac{\sigma(r)}{\sigma(s)}
\iff
\sigma(t)=\sqrt{\sigma(r)\,\sigma(s)}.
\]

On $\gamma$, the map $L\mapsto\sigma(L)$ is continuous and strictly monotone.
Hence, by the intermediate value theorem,
there exists a $t$ such that $\sigma(t)=\sqrt{\sigma(r)\sigma(s)}$,
and by monotonicity it is unique.
\end{proof}

\begin{lemma}[A5(ii): continuity]
For any fixed $r\in\gamma$, the map $s\mapsto\phi_{u,v}(O;r,s)$
is continuous on $\gamma$.
\end{lemma}

\begin{proof}
We have
\[
\phi_{u,v}(O;r,s)
=
\frac{1}{2}\log\frac{\sigma(r)}{\sigma(s)}
=
\frac{1}{2}\bigl(\log\sigma(r)-\log\sigma(s)\bigr).
\]
Since $s\mapsto\sigma(s)$ is continuous and positive on $\gamma$,
the result follows by composition.
\end{proof}

\begin{remark}[Monotonicity and boundary behavior]
As $L$ varies continuously along $\gamma$,
the quantity $\sigma(L)$ varies monotonically
($\sigma\to 0^+$ as $L\parallel U$,
and $\sigma\to +\infty$ as $L\parallel V$).
Consequently,
\[
\phi_{u,v}(O;r,s)=\frac{1}{2}\log\frac{\sigma(r)}{\sigma(s)}
\]
also varies continuously and monotonically with respect to $s$.

At the boundary $S_P$, the angle diverges
($\phi\to \mp\infty$),
which is consistent with the present framework
where the domain is restricted to $D_P$.
\end{remark}

\begin{proposition}[Affine invariance (compatibility with transformations)]
For any affine transformation $T$, we have
\[
\phi_{u,v}(O;A,B)
=
\phi_{T(u),T(v)}\bigl(T(O);T(A),T(B)\bigr).
\]
\end{proposition}

\begin{proof}
An affine transformation $T(x)=Ax+b$ multiplies any signed area by $\det(A)$.
Hence, when the auxiliary line $\Lambda$ is replaced by $T(\Lambda)$,
both the numerator and denominator in
\[
\sigma_\Lambda(L)
=
\frac{[OP_LP_U]}{[OP_LP_V]}
\]
are multiplied by the same factor $\det(A)$.
Therefore,
\[
\sigma_{T(\Lambda)}\bigl(T(L)\bigr)=\sigma_\Lambda(L).
\]

It follows that the area cross ratio $CR_{\mathrm{area}}$ is invariant,
and hence the affine angle $\phi$, defined as its logarithm, is also invariant.
This proves the claim.
\end{proof}

\begin{theorem}[Affine angle as a parabolic degeneration of the Cayley--Klein angle]
\label{thm:affine-angle-CK-degeneration}

In Cayley--Klein geometry, angles are determined by an absolute conic
$Q=0$ in the projective plane.

When this conic degenerates into a product of two linear factors,
\[
Q=(l_1x+m_1y+n_1z)(l_2x+m_2y+n_2z),
\]
the absolute corresponds to two directions
$U_\infty$ and $V_\infty$ on the line at infinity.

In this case, for two rays $OA$ and $OB$ with vertex $O$,
the Cayley--Klein angle is given by
\[
\theta(A,B)
=
\frac{1}{2}\log cr(OA,OB;U_\infty,V_\infty).
\]
\end{theorem}

\begin{remark}
This angle coincides with the $(u,v)$-affine angle defined in
\cref{def:affine-angle-acr}:
\[
\phi_{u,v}(O;A,B)
=
\frac{1}{2}\log
CR_{\mathrm{area}}(L_A,L_B;U,V).
\]
Thus, the affine angle can be understood as a parabolic degeneration
of the Cayley--Klein angle.
\end{remark}

\begin{proof}
As shown in the previous section, the quantity
\[
\sigma_\Lambda(L)
=
\frac{[OP_LP_U]}{[OP_LP_V]}
\]
is independent of the choice of $\Lambda$,
and depends only on the direction of the ray $L$.
It is therefore an affine invariant associated with directions.

Under normalization, we may assume that $U$ is the $x$-axis
and $V$ is the $y$-axis.
Then, writing the ray $L$ as $y=mx$, we obtain
\[
\sigma(L)=c\,m
\qquad (c>0),
\]
where $c$ depends only on $\Lambda$.

Hence,
\[
CR_{\mathrm{area}}(L_A,L_B;U,V)
=
\frac{\sigma(L_A)}{\sigma(L_B)}
=
\frac{m_A}{m_B}.
\]

On the other hand, in projective geometry,
directions of rays correspond to points on the line at infinity,
and the cross ratio
\[
cr(OA,OB;U_\infty,V_\infty)
\]
coincides with the cross ratio of slopes.

Therefore,
\[
\phi_{u,v}(O;A,B)
=
\frac{1}{2}\log
\frac{m_A}{m_B}
\]
agrees with
\[
\theta(A,B)
=
\frac{1}{2}
\log
cr(OA,OB;U_\infty,V_\infty).
\]

This shows that the affine angle introduced in this paper
arises as a parabolic degeneration of the Cayley--Klein angle.
\end{proof}

\section{Transformation Group}

In Cayley--Klein geometry, angles are understood as quantities
that are invariant under suitable subgroups of the projective transformation group.
This viewpoint, which characterizes geometries by their invariant transformation groups,
is based on Klein's Erlangen program.

Since, as shown in the previous section,
the affine angle can be interpreted as a Cayley--Klein angle,
it is natural to determine the subgroup of affine transformations
under which the $(u,v)$-affine angle is invariant.

In this section, we determine the transformation group
that preserves the $(u,v)$-affine angle.
We restrict ourselves to affine transformations that preserve
the reference directions $u,v$.

\begin{proposition}[Determination of the linear part]
Let $A\in GL(2,\mathbb{R})$ be a linear transformation
that preserves the $(u,v)$-affine angle.
Then $A$ preserves the directions $u,v$ as eigen-directions,
and in a suitable coordinate system it is represented as
\[
A=
\begin{pmatrix}
\alpha & 0\\
0 & \beta
\end{pmatrix}
\qquad (\alpha\beta>0).
\]
\end{proposition}

\begin{proof}
From the previous section, the affine angle is given by
\[
\phi_{u,v}(O;A,B)
=
\frac{1}{2}\log\frac{m_A}{m_B},
\]
where $m_A$ and $m_B$ denote the slopes of the rays $OA$ and $OB$.

A linear transformation
\[
(x,y)\mapsto (ax+by,\;cx+dy)
\]
acts on slopes by
\[
m\longmapsto \frac{c+dm}{a+bm}.
\]

For the $(u,v)$-affine angle to be invariant,
the ratio
\[
\frac{m_A}{m_B}
\]
must be preserved for all lines.
That is, we must have
\[
\frac{c+dm_A}{a+bm_A}\Big/\frac{c+dm_B}{a+bm_B}
=
\frac{m_A}{m_B}
\]
for all $m_A,m_B$.

Hence,
\[
b=c=0,
\]
and therefore
\[
A=
\begin{pmatrix}
\alpha & 0\\
0 & \beta
\end{pmatrix}.
\]

Moreover, in order for the affine angle to remain real-valued,
the induced transformation of slopes
\[
m\mapsto \frac{\beta}{\alpha}m
\]
must be a positive scaling.
Thus we require
\[
\frac{\beta}{\alpha}>0,
\]
which is equivalent to
\[
\alpha\beta>0.
\]
\end{proof}

\begin{proposition}[Invariant transformation group]
The full group of affine transformations preserving
the $(u,v)$-affine angle is given by
\[
G_{u,v}\ltimes\mathbb{R}^2,
\]
where $G_{u,v}\cong(\mathbb{R}_{>0})^2$
is the group of linear transformations preserving
the directions $u,v$.
\end{proposition}

\begin{proof}
Translations do not affect directions or ratios of areas,
and hence preserve the affine angle.
Therefore, the full invariant transformation group is
\[
(\mathbb{R}_{>0})^2\ltimes\mathbb{R}^2.
\]

This coincides with the group of affine transformations
whose linear part preserves $u$ and $v$ as eigen-directions.
\end{proof}

\section{Isoptic Sets}

In Euclidean geometry, for two fixed points $A$ and $B$, 
the locus of points from which the segment $AB$ is seen under a constant angle $\theta$ 
(the isoptic curve) is a circle. 
In other words, the isoptic set of the Euclidean angle is given by a circle.

Thus, the study of isoptic sets associated with a given angle quantity 
is a fundamental problem for understanding the geometric structure induced by that angle.

In this section, we investigate the isoptic sets of the $(u,v)$-affine angle 
and show that they are given by hyperbolas. 
Moreover, we characterize the region where the affine angle takes real values: 
it corresponds to the subset of the hyperbola on which $\sigma(L_{PA})$ and $\sigma(L_{PB})$ have the same sign. 
In general, this admissible region consists of both connected components of the hyperbola.

\begin{figure}[htbp]
  \centering
  \begin{minipage}{0.48\textwidth}
    \centering
    \includegraphics[width=\linewidth]{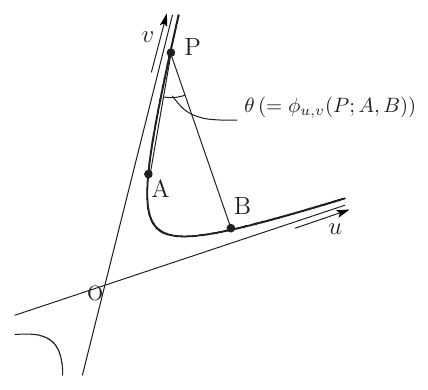}
    \subcaption{}
    \label{fig:affine-angle-isoptic-hyperbola}
  \end{minipage}
\caption{The locus forms a hyperbola; the real-valued affine angle is defined
on the subset where $\sigma(L_{PA})\sigma(L_{PB})>0$, which in general includes
both connected components.}
\end{figure}

\subsection{Affine Plane with Reference Directions}

\begin{definition}[Affine plane with reference directions]
\label{def:affine-plane-UV}
Let $\mathbb A^2$ be the real affine plane.
Fix a point $P\in\mathbb A^2$ and two independent rays
$U,V\in\mathcal R_P$ emanating from $P$.
We call the pair
\[
\Aff^2 := (\mathbb A^2;\, U,V)
\]
an affine plane with reference directions.
\end{definition}

Up to this point, we have denoted by $\phi_{u,v}(O;A,B)$
the $(u,v)$-affine angle between the rays $OA$ and $OB$ with vertex $O$.
From now on, for fixed points $A,B$,
we write $\phi_{u,v}(P;A,B)$ for the corresponding angle at $P$.

\begin{maintheorem}[Shape of the isoptic locus]
\label{mthe:equal-angle-set}
Let $\theta\in \mathbb{R}\setminus\{0\}$ be a constant,
and let $A,B$ be two distinct points in the affine plane.
Then the locus of points $P$ satisfying
\[
\phi_{u,v}(P;A,B)=\theta
\]
is given by a connected component of a hyperbola passing through $A$ and $B$.
Its asymptotes are parallel to the reference directions $u$ and $v$.

Moreover, the locus $\mathcal L_\theta$ is characterized
as the set of points $P$ for which the area cross ratio
\[
\frac{\sigma(L_{PA})}{\sigma(L_{PB})}
\]
is constant.
\end{maintheorem}

\begin{remark}
In Euclidean geometry, the isoptic locus of a segment is a circle,
whereas for the $(u,v)$-affine angle it is a hyperbola.
This difference reflects the fact that the Euclidean angle
is based on rotations,
while the $(u,v)$-affine angle is defined as the logarithm
of an area cross ratio associated with two parallel directions.
\end{remark}

\subsection{Hyperbolic Nature of the Isoptic Locus under Normalization}

We first show that, under normalization, the isoptic condition gives rise to a hyperbola.

\begin{figure}[htbp]
  \centering
 \begin{subfigure}{0.45\textwidth}
    \centering
    \includegraphics[scale=1.0]{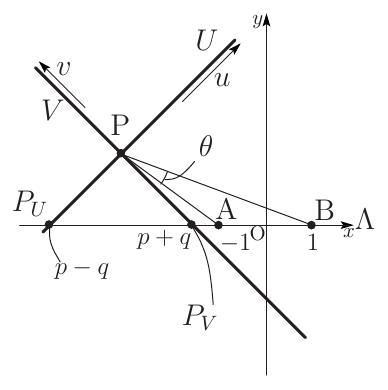}
    \caption{}
    \label{fig:affine-angle-isoptic-proof-normalized}
  \end{subfigure}
  \begin{subfigure}{0.45\textwidth}
    \centering
    \includegraphics[scale=1.0]{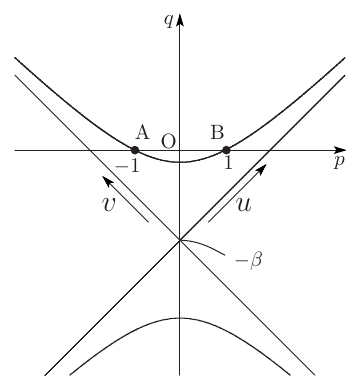}
    \caption{}
    \label{fig:isoptic-normalized-hyperbola-proof}
  \end{subfigure}
\caption{(a) Normalized configuration for the proof of the isoptic property, with fixed reference directions $u$ and $v$ and a point $P=(p,q)$. (b) Normalized isoptic curve obtained in the proof, given by $p^2-(q+\beta)^2=1-\beta^2$.}
\end{figure}

\begin{proposition}[Hyperbolicity of the isoptic locus]
\label{prop:isoptic-hyperbola}
Under normalization, the set of points $P=(p,q)$ satisfying
\[
\phi_{u,v}(P;A,B)=\theta
\]
is given by the hyperbola
\[
p^2-(q+\beta)^2=1-\beta^2,
\]
where $\beta$ is a constant determined by $\theta$.
\end{proposition}

\begin{proof}
By an affine normalization, we may assume without loss of generality that
\[
A=(-1,0),\quad B=(1,0),\quad u=(1,1),\quad v=(1,-1).
\]
Let $P=(p,q)$ and suppose that $\phi_{u,v}(P;A,B)=\theta$.

Set $\lambda:=e^{2\theta}$. Then
\[
\phi_{u,v}(P;A,B)=\theta
\iff
\frac{1}{2}\log\frac{\sigma(L_{PA})}{\sigma(L_{PB})}=\theta
\iff
\frac{\sigma(L_{PA})}{\sigma(L_{PB})}=\lambda.
\]

Consider the lines in the directions $u=(1,1)$ and $v=(1,-1)$:
\[
U:\ y=x-p+q,\qquad
V:\ y=-x+p+q,\qquad
\Lambda:\ y=0.
\]
Then
\[
P_U=(p-q,0),\qquad
P_V=(p+q,0),
\]
and hence
\[
\sigma(L_{PA})=\frac{p-q+1}{p+q+1},\qquad
\sigma(L_{PB})=\frac{p-q-1}{p+q-1}.
\]

Therefore,
\begin{align*}
\frac{\sigma(L_{PA})}{\sigma(L_{PB})}=\lambda
&\iff
\frac{\frac{p-q+1}{p+q+1}}{\frac{p-q-1}{p+q-1}}=\lambda\\
&\iff
\frac{p^2-(q-1)^2}{p^2-(q+1)^2}=\lambda\\
&\iff
p^2-(q-1)^2=\lambda\bigl(p^2-(q+1)^2\bigr).
\end{align*}

Assume $\theta\neq0$, i.e.\ $\lambda\neq1$.
Then we obtain
\[
(1-\lambda)p^2+(\lambda-1)q^2+2(\lambda+1)q+(1-\lambda)=0.
\]

Dividing both sides by $\lambda-1$, we get
\[
-\,p^2+q^2+2\frac{\lambda+1}{\lambda-1}q-1=0.
\]

Let
\[
\beta:=\frac{\lambda+1}{\lambda-1}.
\]
Completing the square yields
\begin{equation}\label{eq:isoptic-hyperbola}
p^2-(q+\beta)^2=1-\beta^2.
\end{equation}

It is easily verified that this curve passes through
$A=(-1,0)$ and $B=(1,0)$.

Since $\theta\neq0$ implies $\lambda=e^{2\theta}\neq1$,
the resulting conic is non-degenerate.
Therefore, in the normalized coordinates $(p,q)$,
the isoptic locus is a hyperbola given by
\[
p^2-(q+\beta)^2=1-\beta^2.
\]
\end{proof}

\begin{corollary}\label{cor:isoptic-hyperbola-coth}
Equation~\cref{eq:isoptic-hyperbola} can be rewritten in the $xy$-plane as
\[
(y+\coth\theta)^2-x^2=\frac{1}{\sinh^2\theta}.
\]
\end{corollary}

\begin{proof}
Since $\lambda=e^{2\theta}$, we have
\[
\beta=\frac{\lambda+1}{\lambda-1}
=\frac{e^{2\theta}+1}{e^{2\theta}-1}
=\frac{e^{\theta}+e^{-\theta}}{e^{\theta}-e^{-\theta}}
=\coth\theta.
\]

Moreover,
\[
1-\beta^2
=
1-\coth^2\theta
=
-\frac{1}{\sinh^2\theta}.
\]
Substituting into \cref{eq:isoptic-hyperbola}, we obtain
\[
(y+\coth\theta)^2-x^2=\frac{1}{\sinh^2\theta}.
\]
\end{proof}

\begin{remark}
From \cref{cor:isoptic-hyperbola-coth},
the center of the hyperbola is $(0,-\coth\theta)$,
and its asymptotes are given by
\[
x=\pm(y+\coth\theta).
\]
\end{remark}

\begin{remark}
The equation
\[
(y+\coth\theta)^2-x^2=\frac{1}{\sinh^2\theta}
\quad\iff\quad
(y\sinh\theta+\cosh\theta)^2-(x\sinh\theta)^2=1
\]
is equivalent to a Minkowski-type hyperbola
\[
Y^2-X^2=\text{const}.
\]
Thus, the point $P=(p,q)$ admits a natural parametrization
in terms of a hyperbolic angle (rapidity):
\[
p=\frac{\sinh t}{\sinh\theta},\qquad
q=\frac{\cosh t}{\sinh\theta}-\coth\theta.
\]
\end{remark}

\begin{remark}
As $\theta\to0$, we have
\[
\sinh\theta\sim\theta,\qquad \coth\theta\sim\frac{1}{\theta},
\]
and the isoptic hyperbola degenerates to the points $A$ and $B$.

On the other hand, as $\theta\to\infty$,
\[
\coth\theta\to1,\qquad \frac{1}{\sinh^2\theta}\to0,
\]
and the hyperbola approaches the pair of lines
\[
p=\pm(q+1).
\]
\end{remark}

\begin{remark}
The isoptic condition
\[
\frac{\sigma(L_{PA})}{\sigma(L_{PB})}=\lambda
\]
means that the area cross ratio associated with the two parallel
direction families $u$ and $v$, as seen from $P$, is constant.

Thus, the isoptic curve is the locus of points that preserve
the ratio of two parallel families of lines,
and the hyperbola naturally arises from this condition.
\end{remark}

\subsection{Isoptic Hyperbolas and Selection of Admissible Regions (Connected Components)}

We now determine which connected component of the hyperbola
obtained in \cref{prop:isoptic-hyperbola}
corresponds to the points $P$ for which
\[
\phi_{u,v}(P;A,B)\in\mathbb{R}.
\]

\begin{definition}[Signed projection quantities and component classification]
\label{def:sigma-sign-component}
Under the above normalization, for a point $P=(p,q)$ we define
\[
\sigma(L_{PA}):=\frac{p-q+1}{p+q+1},\qquad
\sigma(L_{PB}):=\frac{p-q-1}{p+q-1},
\]
whenever the denominators are nonzero.

In what follows, the signs of $\sigma(L_{PA})$ and $\sigma(L_{PB})$
are used to determine to which connected component
of the angle space (divided by the reference directions $U,V$)
the rays $PA$ and $PB$ belong.
\end{definition}

\begin{proposition}[Sign conditions and connected components]
\label{prop:same-component-PA-PB}
Under the normalization
\[
A=(-1,0),\quad B=(1,0),\quad u=(1,1),\quad v=(1,-1),
\]
let $P=(p,q)$ satisfy $p\pm q\neq \pm1$.

Then the following hold:
\begin{itemize}
\item[(i)]
\[
\sigma(L_{PA})>0 \iff |q|<|p+1|.
\]

\item[(ii)]
\[
\sigma(L_{PB})>0 \iff |q|<|p-1|.
\]

\item[(iii)]
\[
\sigma(L_{PA})<0 \iff |q|>|p+1|,\qquad
\sigma(L_{PB})<0 \iff |q|>|p-1|.
\]

\item[(iv)]
The rays $PA$ and $PB$ belong to the same connected component
(with respect to $(U,V)$) if and only if
\[
\sigma(L_{PA})\sigma(L_{PB})>0,
\]
equivalently,
\[
\bigl((p+1)^2-q^2\bigr)\bigl((p-1)^2-q^2\bigr)>0.
\]

\item[(v)]
In particular, $PA$ and $PB$ lie in the positive component if and only if
\[
|q|<|p+1|,\qquad |q|<|p-1|.
\]
\end{itemize}
\end{proposition}

\begin{proof}
Statements (i) and (ii) follow respectively from
\[
\sigma(L_{PA})>0 \iff (p-q+1)(p+q+1)>0 \iff (p+1)^2-q^2>0,
\]
\[
\sigma(L_{PB})>0 \iff (p-q-1)(p+q-1)>0 \iff (p-1)^2-q^2>0.
\]

Statement (iii) follows immediately by taking the negation of (i) and (ii).

For (iv), the claim follows from the fact that two rays belong to the same connected component
if and only if the corresponding values of $\sigma$ have the same sign.

Finally, (v) follows immediately by combining (i) and (ii).
\end{proof}

\begin{proposition}[Admissible connected components]\label{prop:isoptic-branch}
The necessary and sufficient condition for the affine angle
\[
\phi_{u,v}(P;A,B)=\theta
\]
to be defined as a real number is
\[
\sigma(L_{PA})\sigma(L_{PB})>0.
\]

Therefore, among the isoptic set, the subset that actually yields a real-valued angle
coincides with the part of the hyperbola
\[
p^2-(q+\beta)^2=1-\beta^2
\]
on which $\sigma(L_{PA})$ and $\sigma(L_{PB})$ have the same sign.
In general, this admissible subset consists of the two connected components of the hyperbola,
corresponding respectively to
\[
\sigma(L_{PA})>0,\ \sigma(L_{PB})>0
\]
and
\[
\sigma(L_{PA})<0,\ \sigma(L_{PB})<0.
\]
\end{proposition}

\begin{proof}
For $\phi_{u,v}(P;A,B)$ to be defined as a real number, 
the argument of the logarithm in
\[
\phi_{u,v}(P;A,B)=\frac{1}{2}\log\frac{\sigma(L_{PA})}{\sigma(L_{PB})}
\]
must be positive.
Hence the necessary and sufficient condition is
\[
\frac{\sigma(L_{PA})}{\sigma(L_{PB})}>0
\iff
\sigma(L_{PA})\sigma(L_{PB})>0.
\]

On the other hand, by \Cref{prop:same-component-PA-PB},
this is equivalent to the condition that the rays $PA$ and $PB$
belong to the same connected component determined by $(U,V)$.

Therefore, among the hyperbola obtained in \Cref{prop:isoptic-hyperbola},
the subset that actually defines the affine angle
is precisely the part where $\sigma(L_{PA})$ and $\sigma(L_{PB})$ have the same sign.
Since this condition corresponds, in general, to both connected components of the hyperbola,
the admissible set consists of its two branches.
\end{proof}

\begin{figure}[htbp]
  \centering
    \includegraphics[scale=1.0]{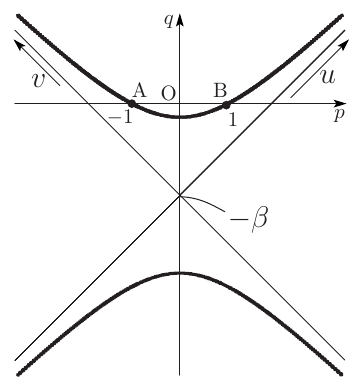}
   \caption{Admissible components of the normalized isoptic hyperbola.
Both connected components correspond to points where the affine angle $\phi_{u,v}(P;A,B)$ is real, characterized by the condition $\sigma(L_{PA})\sigma(L_{PB})>0$.}
    \label{fig:normalized-isoptic-admissible-branch}
\end{figure}

\bigskip

\begin{proof}[Proof of the Main Theorem]
By \cref{prop:isoptic-hyperbola}, the isoptic locus is given by a hyperbola.

Moreover, by \cref{prop:isoptic-branch}, the affine angle is real-valued
precisely on the subset where $\sigma(L_{PA})$ and $\sigma(L_{PB})$ have the same sign.

Therefore, the set of points $P$ satisfying
\[
\phi_{u,v}(P;A,B)=\theta
\]
is given by the admissible part of a hyperbola passing through $A$ and $B$.
In general, this admissible part consists of both connected components of the hyperbola.

Furthermore, by \cref{prop:isoptic-hyperbola,cor:isoptic-hyperbola-coth},
its asymptotes are parallel to the reference directions $u$ and $v$.

This completes the proof.
\end{proof}

\begin{remark}[Isoptic hyperbola and admissible region]
\label{rem:hyperbola-equal-angle}
By \cref{prop:same-component-PA-PB},
the necessary and sufficient condition for $\phi_{u,v}(P;A,B)$
to be real-valued is
\[
\sigma(L_{PA})\sigma(L_{PB})>0.
\]

Hence, among the isoptic locus $\mathcal L_\theta$,
only the points satisfying
\[
((p+1)^2-q^2)((p-1)^2-q^2)>0
\iff
(p^2-q^2+1)^2-4p^2>0
\]
correspond to real-valued affine angles.

In other words, the admissible part of the isoptic hyperbola
is the region separated by the two degenerate hyperbolas
\[
(p+1)^2-q^2=0,\qquad (p-1)^2-q^2=0,
\]
and in general consists of both connected components of the hyperbola.
\end{remark}

\begin{theorem}[Equivalent representations]
The affine angle admits the following equivalent descriptions:
\begin{enumerate}
  \item[(i)] \textbf{Representation via hyperbolic sector area.}  
  Consider a rectangular hyperbola with asymptotes $U$ and $V$.
  Then $\phi_{u,v}(O;A,B)$ coincides with the signed area
  of the hyperbolic sector determined by the rays $OA$ and $OB$.

  \item[(ii)] \textbf{Representation via triangle area ratios.}  
  Let $\Lambda$ be a line intersecting $U$ and $V$,
  and denote the intersection points by $P_U,P_V$.
  Let $A_\Lambda=OA\cap\Lambda$ and $B_\Lambda=OB\cap\Lambda$.
  Writing $[XYZ]$ for the signed area of triangle $XYZ$, we have
  \[
\phi_{u,v}(O;A,B)
=
\tfrac12 \log\!\left(
\frac{[O A_\Lambda P_U]/[O A_\Lambda P_V]}
     {[O B_\Lambda P_U]/[O B_\Lambda P_V]}
\right).
\]
\end{enumerate}
\end{theorem}

\begin{proof}
(i) follows from the fact that the sector area of the rectangular hyperbola
$XY=1$ is given by $\frac{1}{2}\log\frac{x_2}{x_1}$.

(ii) is contained in the argument of
\cref{thm:CR-Lambda-property}, and is therefore omitted.
\end{proof}

\section{The Power Theorem for Hyperbolas}

In this section, we show that the affine angle geometry introduced in this paper
possesses a rich geometric structure beyond isoptic loci,
by establishing a power-type theorem for hyperbolas.

For circles, it is well known that for any line through a fixed point,
the product of the two directed segments determined by the intersections
with the circle is constant (the power of a point theorem).
A corresponding result has also been established for parabolas
within the framework of difference-angle geometry.

Here, we present the hyperbolic analogue of this phenomenon.

In the case of hyperbolas, however, there are several possible candidates
for what should be regarded as the “power.”
In this paper, we adopt a quantity defined via an area-based construction
associated with the asymptotic directions.

Although the two asymptotes of a hyperbola must be treated symmetrically,
we will show that the resulting product is in fact determined
by the data from only one asymptotic direction.
\medskip

First, we specify three requirements that a quantity must satisfy
in order to be regarded as a “power” in this paper.

\begin{convention}[General requirements for a power]
We require that a quantity called the “power” satisfy the following conditions:
\begin{enumerate}
\item For any curve $C$ and point $P$, there exists a line $L$ through $P$
intersecting $C$ at two points $A,B$ (counting tangency as a double point).

\item For $P,A,B$, a quantity is defined in a uniform way,
and the product of the two corresponding values (called the “power”)
is independent of the choice of $L$.

\item The sign of this product classifies the position of $P$
as lying inside, outside, or on the curve $C$.
\end{enumerate}
\end{convention}

\medskip

For computational convenience, we normalize the hyperbola to the rectangular form
\[
y=\frac{\kappa}{x}.
\]

For three points $O=(0,0)$, $A=(x_A,y_A)$, $B=(x_B,y_B)$,
we define the area of the parallelogram spanned by $OA$ and $OB$ as
\[
|x_Ay_B-y_Ax_B|.
\]

\begin{definition}[Hyperbolic core quantity]
\label{def:hyperbolic-core-quantity}
We define
\[
pq-\kappa
\]
and call it the hyperbolic core quantity.
This depends only on the point $P$
and is independent of the choice of lines.
\end{definition}

\begin{lemma}[Normalization to a rectangular hyperbola]
\label{lem:affine-normalize}
Let $H_\theta$ be the hyperbola containing the isoptic locus $\mathcal L_\theta$
obtained in \cref{mthe:equal-angle-set}.
Since its asymptotes are parallel to $u,v$,
there exists an affine transformation together with a scaling
that maps $H_\theta$ to
\[
xy=\kappa \qquad (\kappa\neq0).
\]
\end{lemma}

\begin{definition}[Asymptotic symmetric area]
\label{def:projection-area}
Let $A$ be a point on a hyperbola $\mathcal H$.
From $A$, project onto the two asymptotes $U_1,U_2$
along directions parallel to the other asymptote,
and denote the projection points by $A_1,A_2$.

For a point $P$, let $S(P,A_i)$ denote the area of the parallelogram
spanned by $PA$ and $PA_i$.
We call this the projected area in direction $i$.

We then define the asymptotic symmetric area of the segment $PA$ by
\[
S_{P,A}:=\sqrt{S(P,A_1)S(P,A_2)}.
\]
\end{definition}

\begin{proposition}[One-sided determinacy]
\label{prop:one-sided-determinacy}
The asymptotic symmetric area is determined by the projected area
in only one asymptotic direction.
More precisely,
\[
S_{P,A}S_{P,B}
=
S(P,A_1)S(P,B_1)
=
S(P,A_2)S(P,B_2).
\]
\end{proposition}

\begin{proof}
Since area is changed only by a constant factor under affine transformations,
we may normalize the configuration and assume that
\[
\mathcal H:\ y=\frac{\kappa}{x},\qquad P=(p,q).
\]
Let $\ell$ be a line through $P$ intersecting $\mathcal H$ at
\[
A=(\alpha,\kappa/\alpha),\qquad B=(\beta,\kappa/\beta).
\]
Then
\[
S(P,A_1)=\frac{\kappa}{\alpha}|\alpha-p|,\qquad
S(P,B_1)=\frac{\kappa}{\beta}|\beta-p|
\]
Hence,
\[
S(P,A_1)S(P,B_1)
=
\frac{\kappa^2}{\alpha\beta}|(p-\alpha)(p-\beta)|.
\]

On the other hand,
\[
S_{P,A}
=
\kappa\sqrt{\left|(p-\alpha)\frac{\beta-p}{\alpha\beta}\right|},\qquad
S_{P,B}
=
\kappa\sqrt{\left|(p-\beta)\frac{\alpha-p}{\alpha\beta}\right|}.
\]

Therefore,
\[
S_{P,A}S_{P,B}
=
\frac{\kappa^2}{\alpha\beta}|(p-\alpha)(p-\beta)|.
\]

This proves
\[
S_{P,A}S_{P,B}=S(P,A_1)S(P,B_1).
\]
The argument for the other asymptote is analogous.
\end{proof}

\begin{remark}[Symmetric asymmetry]
Although the asymptotic symmetric area is defined symmetrically
with respect to the two asymptotic directions,
\cref{prop:one-sided-determinacy} shows that the product
is determined by data from only one direction.

Thus the essential quantity is given by
\[
\frac{(p-\alpha)(p-\beta)}{\alpha\beta},
\]
which, up to a constant factor, corresponds to the hyperbolic core quantity
\[
pq-\kappa.
\]
\end{remark}

\begin{figure}[htbp]
 \begin{subfigure}{0.50\textwidth}
    \includegraphics[scale=0.8]{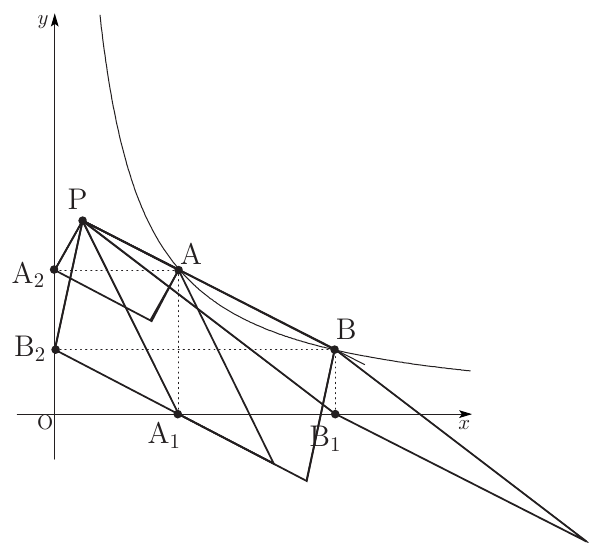}
    \caption{}
    \label{fig:hyperbolic-power-two-sided}
  \end{subfigure}
\hfill
  \begin{subfigure}{0.50\textwidth}
    \includegraphics[scale=0.8]{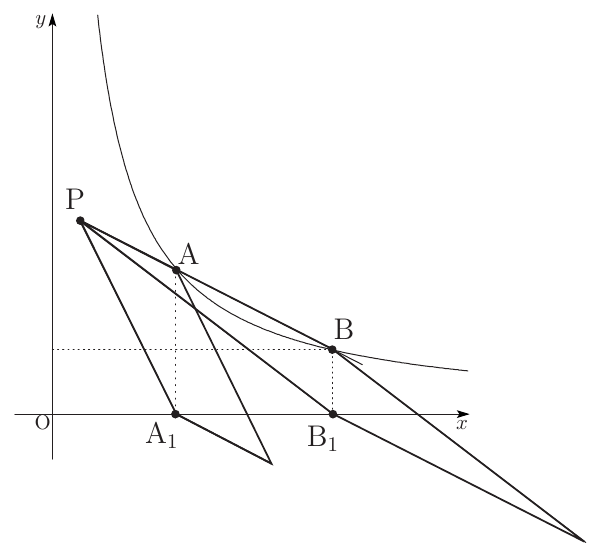}
    \caption{}
    \label{fig:hyperbolic-power-one-sided}
  \end{subfigure}
\caption{(a) Symmetric construction of the hyperbolic power using both asymptotic directions.
(b) Reduced one-sided construction.
The product $S(P,A_1)S(P,B_1)$ is independent of the secant $PAB$
and determines the hyperbolic power.}
\end{figure}

\begin{theorem}[Power theorem for hyperbolas]
\label{thm:hyperbolic-power}
Let
\[
\mathcal H:\ y=\frac{\kappa}{x}
\]
be a hyperbola, and let $P=(p,q)$ be a point.
Let $\ell$ be a line through $P$ intersecting $\mathcal H$ at two points
\[
A=(\alpha,\kappa/\alpha),\qquad
B=(\beta,\kappa/\beta).
\]

Let $S_{P,A}$ and $S_{P,B}$ denote the asymptotic symmetric areas
of the segments $PA$ and $PB$, respectively.
Then the product
\[
S_{P,A}S_{P,B}
\]
is independent of the choice of the line $\ell$, and is given by
\[
S_{P,A}S_{P,B}
=
\kappa\,|pq-\kappa|.
\]
\end{theorem}

\begin{remark}
Although the formulation in \cref{thm:hyperbolic-power} is symmetric,
its structure is essentially one-sided:
the value is completely determined by a single asymptotic direction.

This can be seen by comparing
\cref{fig:hyperbolic-power-two-sided} and
\cref{fig:hyperbolic-power-one-sided}.
\end{remark}

\begin{proof}
Let the intersections of $\ell$ and $\mathcal H$ be
$A=(\alpha,\kappa/\alpha)$ and $B=(\beta,\kappa/\beta)$.
Then $\ell$ is given by
\[
\ell:\ y=\frac{\kappa(\alpha+\beta-x)}{\alpha\beta}.
\]

Since $P=(p,q)$ lies on $\ell$, we have
\[
q=\frac{\kappa(\alpha+\beta-p)}{\alpha\beta}.
\]
Thus,
\[
\frac{pq\alpha\beta}{\kappa}
=
(\alpha+\beta)p-p^2.
\]

By \cref{prop:one-sided-determinacy}, we have
\[
S_{P,A}S_{P,B}
=
S(P,A_1)S(P,B_1)
=
\frac{\kappa^2}{\alpha\beta}|(p-\alpha)(p-\beta)|.
\]

Hence,
\begin{align*}
S_{P,A}S_{P,B}
&=
\frac{\kappa^2}{\alpha\beta}|(p-\alpha)(p-\beta)|\\
&=
\frac{\kappa^2}{\alpha\beta}|p^2-(\alpha+\beta)p+\alpha\beta|\\
&=
\frac{\kappa^2}{\alpha\beta}\left|-\frac{1}{\kappa}pq\alpha\beta+\alpha\beta\right|\\
&=
\kappa|pq-\kappa|.
\end{align*}
\end{proof}

We call this value the \emph{hyperbolic power} of the point $P$
with respect to $\mathcal H$, and denote it by $\Pi(P;\mathcal H)$.
It follows that the hyperbolic power satisfies all three requirements
stated above.

\begin{remark}
The sign of the quantity
\[
pq-\kappa
\]
determines on which side of the hyperbola
\[
\mathcal H:\ xy=\kappa
\]
the point $P=(p,q)$ lies.
In particular,
\[
pq-\kappa>0,\qquad pq-\kappa<0
\]
divide the plane into two regions.

As in the case of circles, this sign provides a classification
of the position of $P$ relative to $\mathcal H$.

In particular, when
\[
pq-\kappa>0,
\]
there exist real tangent lines from $P$ to $\mathcal H$.
\end{remark}

\begin{theorem}[Hyperbolic radical axis]
\label{thm:hyperbolic-radical-axis}
Let $\mathcal{H}$ and $\mathcal{I}$ be two hyperbolas whose asymptotes are parallel to $u$ and $v$,
and suppose that they intersect at two distinct points $A$ and $B$.
Then for any point $P$ on the line $\ell_{AB}$ containing the common chord $AB$, we have
\[
\Pi(P;\mathcal{H})=\Pi(P;\mathcal{I}).
\]
The line $\ell_{AB}$ is called the \emph{hyperbolic radical axis} of $\mathcal{H}$ and $\mathcal{I}$.
\end{theorem}

\begin{proof}
The statement is trivial for $P=A,B$, since both powers vanish.

Assume $P\neq A,B$.
Under normalization, both hyperbolas take the form $y=\frac{\kappa}{x}$,
and the intersection points $A,B$ coincide for $\mathcal{H}$ and $\mathcal{I}$.

By \cref{thm:hyperbolic-power}, we have
\[
\Pi(P;\mathcal{H})
=
S_{P,A}S_{P,B}
=
\Pi(P;\mathcal{I}).
\]
\end{proof}

\begin{remark}[Degenerate cases]
If the hyperbolas are tangent at a point, the radical axis is the common tangent at that point.

If the two hyperbolas coincide, then the powers coincide for all points,
and the radical axis is not defined, as in the case of circles.
\end{remark}

\begin{figure}[htbp]
  \centering

  \begin{subfigure}{0.32\textwidth}
    \centering
    \includegraphics[width=\textwidth]{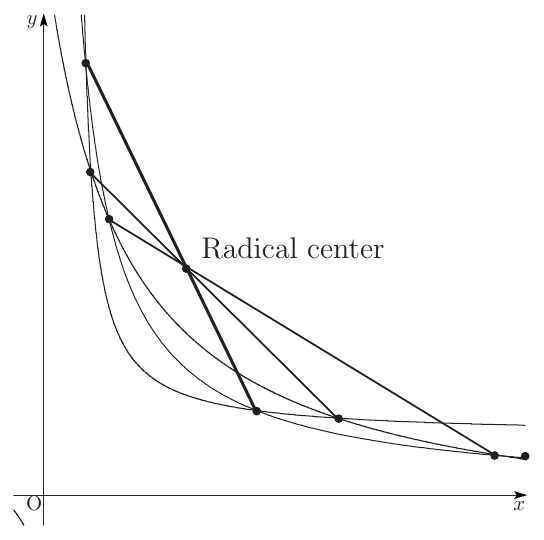}
    \caption{All three hyperbolas lie on the same branch.}
  \end{subfigure}
  \hfill
  \begin{subfigure}{0.32\textwidth}
    \centering
    \includegraphics[width=\textwidth]{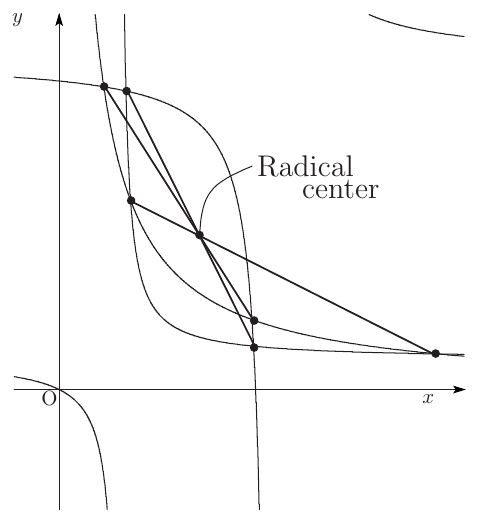}
    \caption{Two hyperbolas share a branch, while the third lies on the opposite branch.}
  \end{subfigure}
  \hfill
  \begin{subfigure}{0.32\textwidth}
    \centering
    \includegraphics[width=\textwidth]{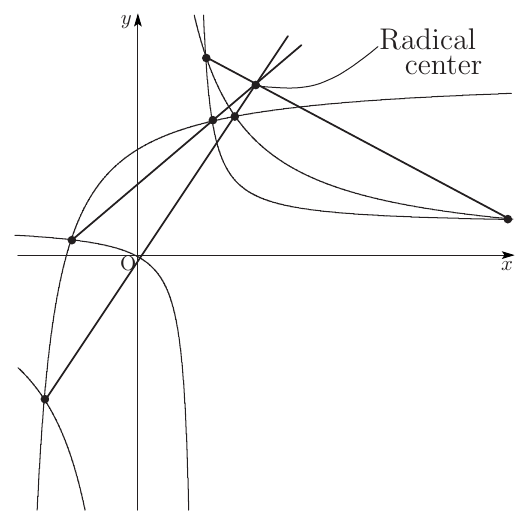}
    \caption{Two hyperbolas belong to one family, while the third belongs to the conjugate family.}
  \end{subfigure}

  \caption{
Hyperbolic radical center.
For three hyperbolas whose asymptotes are parallel to the reference directions,
their hyperbolic radical axes are concurrent at a single point.
These three figures illustrate all possible configurations of branches and families (up to relabeling).
  }

  \label{fig:Hyperbolic-radical-center}
\end{figure}

\begin{theorem}[Hyperbolic radical center]
\label{thm:hyperbolic-radical-center}
Let $\mathcal{H}, \mathcal{I}, \mathcal{J}$ be three hyperbolas whose asymptotes are parallel to $u$ and $v$.
Then the three hyperbolic radical axes
\[
\ell_{\mathcal{H}\mathcal{I}},\quad
\ell_{\mathcal{I}\mathcal{J}},\quad
\ell_{\mathcal{J}\mathcal{H}}
\]
are concurrent.
Their intersection point is called the \emph{hyperbolic radical center}.
\end{theorem}

\begin{proof}
Let $P\in\ell_{\mathcal{H}\mathcal{I}}\cap\ell_{\mathcal{I}\mathcal{J}}$.
Then
\[
\Pi(P;\mathcal{H})=\Pi(P;\mathcal{I}),\qquad
\Pi(P;\mathcal{I})=\Pi(P;\mathcal{J}).
\]
Hence $\Pi(P;\mathcal{H})=\Pi(P;\mathcal{J})$,
so $P\in\ell_{\mathcal{J}\mathcal{H}}$.
Therefore, the three lines are concurrent.
\end{proof}

In contrast to the focal equation in difference-angle geometry,
which characterizes the parabola,
the main theorem of the parabolic power provides its quantitative aspect.
The correspondence between these two viewpoints
forms the foundation of the parabolic structure in difference-angle geometry.

\appendix
\section{On the Structural Aspects of Affine Angle Geometry}

In this appendix, we present several propositions concerning hyperbolas
that arise by analogy with Euclidean and difference-angle geometries.
They can be viewed as hyperbolic counterparts of the area invariance
associated with equally spaced configurations on circles and parabolas,
where the spacing is logarithmic rather than linear.

\subsection{Area invariance under geometric progression}

\begin{figure}[htbp]
  \centering
 \begin{subfigure}{0.45\textwidth}
    \centering
    \includegraphics[scale=1.0]{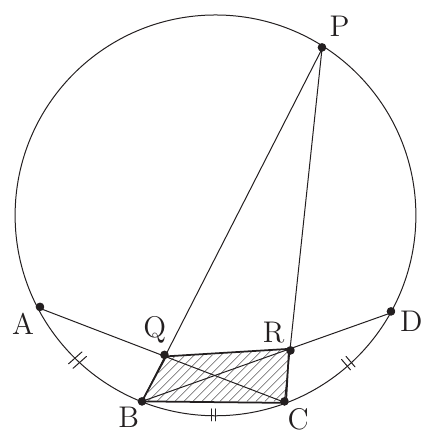}
    \caption{}
    \label{fig:area-invariance-circle}
  \end{subfigure}
  \begin{subfigure}{0.45\textwidth}
    \centering
    \includegraphics[scale=1.0]{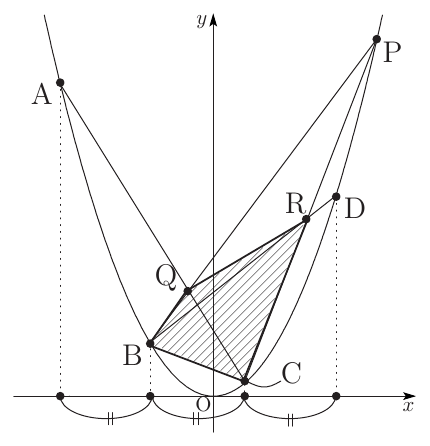}
    \caption{}
    \label{fig:area-invariance-parabola}
  \end{subfigure}
\caption{Corresponding configurations on the circle and the parabola.
(a) Equal-arc subdivision on the circle.
(b) Equal difference-angle norm subdivision on the parabola.
In both cases, the area of the quadrilateral $BCRQ$ is invariant
with respect to the position of $P$.}
\end{figure}

The following problem was originally posed by the author
well before the development of difference-angle geometry,
based on analogies between circles and parabolas.

\begin{proposition}
Let $A,B,C,D$ be four distinct points on a circle such that
\[
\overset{\frown}{AB}=\overset{\frown}{BC}=\overset{\frown}{CD}.
\]
Let $P$ be a point on the arc $\overset{\frown}{AD}$ not containing $B$.
Let $Q=AC\cap PB$ and $R=BD\cap PC$.
Then the area of the quadrilateral $BCRQ$ is independent of the position of $P$.
\end{proposition}

The corresponding prototype is the following.
Here, $\dnorm{\cdot}$ denotes the difference-angle norm.

\begin{proposition}
Let $A,B,C,D$ be four distinct points on the parabola
\[
\mathcal P:\ y=\kappa x^2
\]
such that
\[
\dnorm{AB}=\dnorm{BC}=\dnorm{CD}.
\]
Let $P$ be a point on $\mathcal P$ excluding the arc $\overset{\frown}{AD}$ containing $B$.
Let $Q=AC\cap PB$ and $R=BD\cap PC$.
Then the area of the quadrilateral $BCRQ$ is independent of the position of $P$.
\end{proposition}

\begin{remark}
Since this problem was originally intended for high school students,
the area is considered in the Euclidean sense in $\mathbb{R}^2$.
However, this quantity can be justified within difference-angle geometry.
\end{remark}

The above two problems are originally published in
\emph{University Mathematics} (Tokyo Publishing),
and the copyright belongs to the publisher.
We now present a new result in the hyperbolic setting.
Here, $x_P$ denotes the $x$-coordinate of a point $P$.

\begin{figure}[htbp]
  \centering
    \includegraphics[scale=1.0]{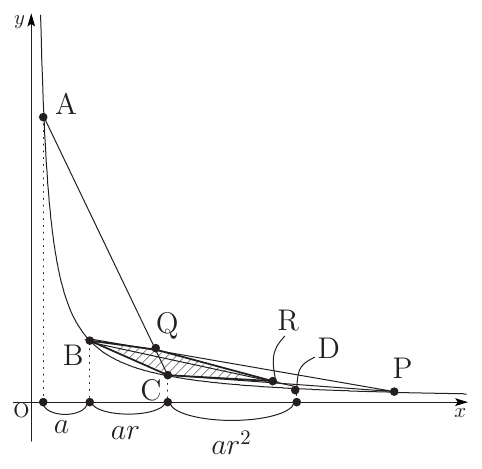}
   \caption{Hyperbolic analogue of the area invariance configuration.
The points $A,B,C,D$ form a geometric progression on $\mathcal H: xy=\kappa$,
corresponding to equal arcs on the circle and equal difference-angle norms on the parabola.
In all cases, the area of the quadrilateral $BCRQ$ is invariant under the motion of $P$.}
    \label{fig:area-invariance-hyperbola}
\end{figure}

\begin{proposition}[Hyperbolic version: area invariance under geometric progression]
\label{prop:hyperbola-area-invariance}
Let $a,r>0$ with $r\neq1$.
On the hyperbola
\[
\mathcal H:\ xy=\kappa \qquad (x>0),
\]
take four distinct points $A,B,C,D$ such that
\[
x_A=a,\quad x_B=ar,\quad x_C=ar^2,\quad x_D=ar^3.
\]
Let $P$ be a point on $\mathcal H$ excluding the arc $\overset{\frown}{AD}$ containing $B$.
Define $Q:=AC\cap PB$ and $R:=BD\cap PC$.

Then the area of the quadrilateral $BCRQ$ is independent of the position of $P$.
\end{proposition}

\begin{proof}
Since areas are scaled only by a constant factor under affine transformations,
we may assume $\kappa=1$, so that $\mathcal H:xy=1$.

We parametrize points on the hyperbola by
\[
X(t):=\Bigl(t,\frac{1}{t}\Bigr)\qquad (t>0).
\]
Then
\[
A=X(a),\quad B=X(ar),\quad C=X(ar^2),\quad D=X(ar^3),\quad P=X(p)\ (p>0).
\]

\medskip
\noindent
\textbf{Lemma 1 (Equation of a chord).}
For $t_1\neq t_2$, the line through $X(t_1)$ and $X(t_2)$ is given by
\[
y=\frac{(t_1+t_2)-x}{t_1t_2}.
\]
\emph{Proof.} This follows immediately from the two-point form. \qed

\medskip
Thus each chord is described by the sum $S=t_1+t_2$ and product $T=t_1t_2$:
\[
L(t_1,t_2):\quad y=\frac{S-x}{T}.
\]

\medskip
\noindent
\textbf{Lemma 2 (Intersection of two chords).}
Let
\[
L(t_1,t_2)\cap L(t_3,t_4)=:Z.
\]
Then
\[
x_Z=\frac{(t_3+t_4)\,t_1t_2-(t_1+t_2)\,t_3t_4}{t_1t_2-t_3t_4}.
\]
\emph{Proof.} Solve $\frac{S_{12}-x}{T_{12}}=\frac{S_{34}-x}{T_{34}}$. \qed

\medskip
Using this, we compute $Q$ and $R$.

For
\[
Q=L(a,ar^2)\cap L(p,ar),
\]
we obtain
\[
x_Q=\frac{a^2r^2+arp-ar^2p-ap}{ar-p},
\]
hence
\[
x_C-x_Q=\frac{a(ar^2-p)(r-1)}{ar-p}.
\]

Similarly, for
\[
R=L(ar,ar^3)\cap L(p,ar^2),
\]
we obtain
\[
x_R=\frac{a^2r^4+ar^2p-ar^3p-arp}{ar^2-p},
\]
hence
\[
x_B-x_R=\frac{a r^3(ar-p)(r-1)}{ar^2-p}.
\]

\medskip
We now compute the area structurally.
Since $Q,C$ lie on $AC$ and $R,B$ lie on $BD$,
we have
\[
[\,BCRQ\,]=\frac{1}{2}\,|\overrightarrow{QC}\times\overrightarrow{RB}|.
\]

Taking direction vectors
\[
\vec d_{AC}=(1,-\tfrac{1}{a^2r^2}),\qquad
\vec d_{BD}=(1,-\tfrac{1}{a^2r^4}),
\]
we obtain
\[
\overrightarrow{QC}=(x_C-x_Q)\vec d_{AC},\qquad
\overrightarrow{RB}=(x_B-x_R)\vec d_{BD}.
\]

Thus
\[
[\,BCRQ\,]
=
\frac{1}{2}\,|x_C-x_Q|\,|x_B-x_R|\cdot|\vec d_{AC}\times\vec d_{BD}|.
\]

A direct computation gives
\[
\vec d_{AC}\times\vec d_{BD}
=
\frac{r^2-1}{a^2r^4}.
\]

On the other hand,
\[
|x_C-x_Q|\,|x_B-x_R|
=
a^2r^3(r-1)^2,
\]
where the dependence on $p$ cancels.

Therefore,
\[
[\,BCRQ\,]
=
\frac12\cdot a^2r^3(r-1)^2\cdot \frac{|r^2-1|}{a^2r^4}
=
\frac{(r+1)|r-1|^3}{2r^2},
\]
which is independent of $p$.
\end{proof}

\begin{remark}[Structural interpretation]
Taking $x_A,x_B,x_C,x_D$ in geometric progression
corresponds to equal spacing of $\log t$ in the parameter $t=x$.

Thus, this result can be viewed as a hyperbolic analogue of
area invariance arising from:
\begin{itemize}
\item equal arc-length spacing on a circle,
\item linear spacing in the difference-angle norm on a parabola,
\item logarithmic spacing on a hyperbola.
\end{itemize}
\end{remark}

\section{First-Order Degenerate Limit of the Cayley--Klein Angle}

In this appendix, we establish that the difference angle arises as a first-order
degenerate limit of the Cayley--Klein angle.
This corresponds to a linear-type degeneration, in contrast to the logarithmic
degeneration leading to the affine angle.
A detailed geometric analysis will be given elsewhere.

\vspace{1ex}

In Cayley--Klein geometry, consider the isotropic directions
\[
m_{t,1}=\frac{1}{t}, \qquad
m_{t,2}=-\frac{1}{t}
\]
with $t\to 0$.
For two rays $\ell_1,\ell_2$ with slopes $m_1,m_2\neq0$,
we consider the cross ratio
\[
\Cr^\vee(\ell_1,\ell_2;I_{t,1},I_{t,2})
=
\Cr\!\left(
\frac{1}{m_1},\frac{1}{m_2};
t,-t
\right).
\]

\begin{lemma}[First-order expansion]
\[
\log\Cr^\vee(\ell_1,\ell_2;I_{t,1},I_{t,2})
=
-2(m_1-m_2)t + O(t^3).
\]
\end{lemma}

\begin{proof}
A direct computation yields
\[
\Cr^\vee
=
\frac{(1-m_1t)(1+m_2t)}
     {(1+m_1t)(1-m_2t)}.
\]
Taking logarithms and using $\log(1+x)=x+O(x^2)$,
the $O(t^2)$ terms cancel by symmetry, and we obtain
\[
\log\Cr^\vee
=
-2(m_1-m_2)t + O(t^3).
\]
\end{proof}

To obtain a finite limit, define
\[
\alpha(t)=-2t.
\]
Then
\[
\lim_{t\to0}
\frac{1}{\alpha(t)}
\log\Cr^\vee(\ell_1,\ell_2;I_{t,1},I_{t,2})
=
m_1-m_2.
\]

Thus, as a first-order degeneration of the Cayley--Klein angle,
we recover the difference of slopes
\[
\Pangle(\ell_1,\ell_2)=m_1-m_2.
\]
In particular, the difference angle arises naturally
as a linear degeneration of the Cayley--Klein angle.

%%%%%%%%%%%%%%%%%%%%%%%%%%%%%%%%%%%%%%%%%%%%%%%%%%%%%%%%%%%

\bibliographystyle{plain}
\bibliography{affine_angle_ref}

\end{document}